\documentclass[11pt,letterpaper]{amsart}

\numberwithin{equation}{section}
\theoremstyle{plain}
\newtheorem{theorem}{Theorem}[section]

\newtheorem{lemma}[theorem]{Lemma}
\newtheorem{conjecture}{Conjecture}

\theoremstyle{definition}

\newtheorem{remark}[theorem]{Remark}

\newcommand{\Rmnum}[1]{\expandafter\@slowromancap\romannumeral #1@}

\newcommand{\mr}{\mathbb{R}}
\newcommand{\ud}{\mathrm{d}}
\allowdisplaybreaks

\address{M. Li, School of Mathematics, Nanjing University, China; Department  of Mathematics, The  Chinese University of Hong Kong, Shatin, NT, Hong Kong}
\email{limx@smail.nju.edu.cn}
\address{X. Xu,   School of Mathematics, Nanjing University, China}
\email{matxuxw@nju.edu.cn}

\begin{document}
	\title{On positivity of the Q-curvatures of  conformal metrics}
	\author{Mingxiang Li,  Xingwang Xu}
	\date{}
	\maketitle
	\begin{abstract}
		We  mainly show  that  for a conformal metric  $g=u^{\frac{4}{n-2m}}|dx|^2$ on $\mathbb{R}^n$ with $n\geq 2m+1$, if the $2m-$order Q-curvature $Q^{(2m)}_g$  is positive and  has  slow decay barrier  near infinity, the lower order Q-curvature  $Q^{(2)}_g$ and $Q^{(4)}_g$ are both positive if $m$ is at least two. 
	\end{abstract}

	\section{Introduction}

	Consider a smooth conformal metric $g=u^{\frac{4}{n-2m}}|dx|^2$ on $\mr^n$ where $m$ is a positive integer satisfying
	$1\leq m<\frac{n}{2}.$ The $2m-$order Q-curvature $Q^{(2m)}_g$  is defined by the following equation:	
	\begin{equation}\label{Q^2m}
		(-\Delta)^mu=Q^{(2m)}_gu^{\frac{n+2m}{n-2m}}.
	\end{equation}
	Thus, the lower order Q-curvature  $Q^{(2k)}_g$ of the metric $g$ with $k$ ($1\leq k\leq m-1$) can be calculated through the equation
	\begin{equation}\label{Q^2k}
		(-\Delta)^{k}(u^{\frac{n-2k}{n-2m}}) = Q^{(2k)}_g u^{\frac{n+2k}{n-2m}}.
	\end{equation}
	In particular, the scalar curvature $R_g$ of $g$ is equal to $Q^{(2)}_g$.
	
	For a given real number $s \in \mr$, we say the $2m-$order Q-curvature $Q_g^{(2m)}$ has {\it slow decay barrier with rate} $s$ at infinity if there exists a constant $c_0 > 0$ such that for $|x|$ sufficiently large, the inequality $Q_g^{(2m)} \geq c_0 |x|^s$ holds true.

	Our main purpose of the current article is to show the following statement.
	\begin{theorem}\label{thm: positive Q implies positve R}
		Given a smooth conformal metric $g=u^{\frac{4}{n-2m}}|dx|^2$ on $\mr^n$ with $2\leq m<\frac{n}{2}$.
		If the positive $2m-$order Q-curvature $Q^{(2m)}_g$  has slow decay barrier with rate  $-2m<s\leq 0$ at infinity, then the scalar curvature $R_g$ and $4-$order  Q-curvature $Q_g^{(4)}$ of the metric $g$ are both positive. 
	\end{theorem}

	Due to our technical limitation, we are unable to derive more. However, many evidences make us to believe that all other lower order Q-curvatures are positive and we state it as a conjecture.
	
	\begin{conjecture}\label{conj}
		Given  a smooth conformal metric $g=u^{\frac{4}{n-2m}}|dx|^2$ on $\mr^n$ with $2\leq m<\frac{n}{2} $.
		If  positive $2m-$order Q-curvature $Q^{(2m)}_g$ has slow decay barrier with the rate $-2m < s \leq 0$ at infinity,
		then,  as long as $1\leq k\leq m-1$,  there holds
		$$Q^{(2k)}_g> 0.$$
	\end{conjecture}
	\begin{remark}
		Our main Theorem \ref{thm: positive Q implies positve R} stated above just confirms the conjecture for $m=2,3.$
	\end{remark}

	Now it is right place to recall some reasons to study such a problem. First of all, for $m=1$, the equation \eqref{Q^2m} is known as the prescribing scalar curvature equation and many research works have been done, just list a few, \cite{CGS},  \cite{GNN}, \cite{GS1}, \cite{GS2}. For $m\geq 2$, from the point view of differential equations,  this is just the higher order semi-linear elliptic type equations which has many applications in physic such as membrane etc.. However, its geometric prescription is relatively new, initiated by Paneitz who derived a fourth order conformal covariant operator. And Q-curvature also appears in the log determinant for Laplace operator under conformal deformation. Due to those geometric or physical properties, it stimulates its active study in recent years.  A lot of  works made the subject very promising including \cite{JX}, \cite{Lin}, \cite{WX}, \cite{Xu} etc.
	
	From analytic perspective, the equation \eqref{Q^2m} is known as  Lane-Emden equation if $Q^{(2m)}_g=1$. By now, Lane-Emden equation has been well studied and many properties of the solution have been derived, the interesting readers are referred to \cite{FWX},  \cite{Lin}, \cite{NY}, \cite{Scouplet},  \cite{WX}  and the references therein. The key breakthrough of the subject is to show that the Lane-Emden equation itself implies that $(-\Delta)^k u > 0$ for all $1 \leq k \leq m-1$ if $u > 0$. The argument for this is based on the spherical average growth estimate by the contradiction argument which cannot extend to the non-constant Q-curvature. Thus it is natural to extend it to other cases so that the maximum principle works for the general higher order equations.
	
	Along this direction, if $(M^n,g)$ is compact of dimension $n\geq 5$, Gursky and Malchiodi \cite{GM} showed that the strong maximum principle holds for Paneitz-Branson operator if  Q-curvature and scalar curvature are both positive for $m =2$.  Based on such a strong maximum principle,  the existence of conformal metric with constant Q-curvature follows. Similar result is also obtained by Hang and Yang \cite{HY}.  More details about Q-curvature and Paneitz-Branson operator can be found in \cite{Bre}, \cite{CHY}, \cite{FG}, \cite{GM}, \cite{HY}  and the references therein.
	
	We try to understand the non-compact case with much simpler topology, namely, work on $\mr^n$ with general conformal metric. When $m = 2$ and fourth Q-curvature $Q_g^{(4)} = |x|^a$, it has been discussed in recent article \cite{FWX} by using a very technical iteration argument and it is very hard to see if it can be generalized to bigger $m$ case. In some sense, our main result above is the first step to reach the desired estimate. 
	
	Before we explain our method, let us briefly recall what we already had in existing literatures. As just mentioned, in \cite{FWX}, Fazly, Wei and the second author provide a desired estimate for $m=2$ case with some special Q-curvature by development of a Moser type iteration technique.   For higher order cases ($m\geq 3$),  Ng\^o and Ye \cite{NY} treated the equation \eqref{Q^2m} for special Q-curvature $Q^{(2m)}_g=|x|^\sigma$ with $\sigma>-2m$ by observing that the potential theory can be applied. With this observation, they showed that the analytical property of solution indeed still valid, i.e. 
	$(-\Delta)^ku>0$ where $1\leq k\leq m-1.$ A similar result is also obtained in \cite{AGHW}. However the geometric information is still missing.

	Now we would like to explain the slow decay barrier condition briefly. From analytic perspective, for $Q^{(2m)}_g=|x|^{-2m}$,  non-existence of the positive solutions to  equation \eqref{Q^2m} is established in \cite{FWX}, \cite{NY} and \cite{Scouplet}. Of course, if $Q^{(2m)}_g\equiv 0$, the solution space is too large and it is hard to get useful geometric information except trivial statement that they are all $m$th-polyharmonic functions. The slow decay barrier condition is used to get rid of this trivial case and try to get the geometric information of the conformal metric to some degree.
	
	The main idea is to show that, under the slow decay barrier assumption, the potential theory for the equation is still true so that we can transfer the differential equation to integral equation with standard fundamental solution for $(-\Delta)^m$ with non-linear measure $Q_g^{(2m)} u^{\frac{n+2m}{n-2m}}dx$. This measure is good enough for us to take derivatives under the integral sign. Thus the work is to show that suitable integrals can be used to express the scalar curvature as well as the $4-$order  Q-curvature $Q_g^{(4)}$ with coefficients in terms of $n$ and $m$ and then we can check those coefficients are non-negative in different cases in terms of relations between $m$ and $n$. 
	
	This paper is organized as follows. In Section \ref{section2}, we give some necessary growth estimate for a non-negative solution $u$ and obtain an integral representation. With help of such integral estimate,  some useful identities are established in Section \ref{section3}. Finally, in Section \ref{section4}, we provide the detailed computations and complete the proof of our claim.
	
	\medskip
	{\bf Acknowledgements.}  The first author would like to thank Professor Juncheng Wei and Professor Dong Ye for helpful discussions. Both authors are  supported by NSFC (No.12171231). We deeply appreciate the valuable comments and suggestions provided by the anonymous referees. 
	
	\section{Integral representation}\label{section2}
	
	We consider a slightly more general case. Suppose that 
	a positive function $u$  satisfies the equation
	\begin{equation}\label{equ:Lane-Emden equ}
		(-\Delta)^mu=Q u^p,\quad \mathrm{in}\;\;\mr^n
	\end{equation}
	where $m$ is an integer satisfying $2\leq m<\frac{n}{2}$, $p>1$,   $Q(x)$ is a given positive function satisfying the slow decay barrier condition at infinity with the rate $$-2m< s \leq 0.$$  We assume that both $Q$ and $u$ belong to $ L^\infty_{loc}(\mr^n)$ and
	\begin{equation}\label{range p}
		p>\frac{n+s}{n-2m}.
	\end{equation}
	In particular,    $p=\frac{n+2m}{n-2m}$  satisfies  the condition \eqref{range p}.
	
	The following lemmas in this section state and prove some useful properties for non-negative solutions of  above equation \eqref{equ:Lane-Emden equ}. For simplicity, we refer $B_R(x)$ as a Euclidean ball in $\mr^n$ with radius $R$ and  center  at  point $x\in \mr^n$. A ball with radius $R$ centered  at origin is simply denoted by $B_R$. Let $C$ be a constant which may be different from line to line.	  The first property can be stated as follows. 
	
	\begin{lemma}\label{lem: B_R Qup}
		For $R\gg1$, there holds
		$$\int_{B_R}Qu^p\ud x\leq CR^{n-2m-\frac{2m+s}{p-1}}.$$
	\end{lemma}
	\begin{proof}
		Choose a smooth cut-off function $\psi$ satisfying $\psi\equiv 1$ in $B_1$ and vanishes outside $B_2$.
		Set
		$$\phi_R(x)=\psi(\frac{x}{R})^q$$
		where $q=\frac{2mp}{p-1}$. A direct computation yields that
		\begin{equation}\label{Dela^m phi_R}
			|\Delta^m\phi_R(x)|\leq   CR^{-2m}\|\psi(x) \|_{C^{2m}(B_2(0))}[\psi(\frac{x}{R})]^{q-2m}\leq CR^{-2m}\phi_R^{\frac{1}{p}}(x).
		\end{equation}
		Making use of the equation \eqref{equ:Lane-Emden equ}, the estimate \eqref{Dela^m phi_R} and integration by parts, one has 
		\begin{align*}
			\int_{\mr^n}Qu^p\psi_R\ud x=&\int_{\mr^n}u(-\Delta)^m\phi_R\ud x\\
			\leq &\int_{B_{2R}\backslash B_R}u|(-\Delta)^m\phi_R|\ud x\\
			\leq &CR^{-2m}\int_{B_{2R}\backslash B_R}u\phi_R^{\frac{1}{p}}\ud x.
		\end{align*}
		On the other hand, by the slow decay barrier assumption $Q(x)\geq C|x|^s$ near infinity, for $R\gg1$, there holds
		\begin{align*}
			\int_{B_{2R}\backslash B_R}u\phi_R^{\frac{1}{p}}\ud x	\leq &CR^{-\frac{s}{p}}\int_{B_{2R}\backslash B_R}Q^{\frac{1}{p}}u\phi_R^{\frac{1}{p}}\ud x\\
			\leq & CR^{n(1-\frac{1}{p})-\frac{s}{p}}\left(\int_{B_{2R}\backslash B_R}Qu^p\phi_R\ud x\right)^{\frac{1}{p}}\\
			\leq & CR^{n(1-\frac{1}{p})-\frac{s}{p}}\left(\int_{\mr^n}Qu^p\phi_R\ud x\right)^{\frac{1}{p}}.
		\end{align*}
		Combining these two estimates, one has
		$$\int_{\mr^n}Qu^p\psi_R\ud x
		\leq CR^{n-2m-\frac{2m+s}{p-1}}.$$
		Finally, using the facts $\psi_R=1$ on $B_R$ as well as both $Q$ and  $u$ are positive, we obtain the estimate  
		$$
		\int_{B_R}Qu^p\ud x\leq \int_{\mr^n}Qu^p\psi_R\ud x\leq CR^{n-2m-\frac{2m+s}{p-1}}
		$$
		which is the desired one we claimed. 
	\end{proof}
	
	Based on the growth of $Q u^p$ over $B_R$ as stated in Lemma \ref{lem: B_R Qup}, we are able to show that $Q u^p$  convoluted with  a suitable  power of $|x-y|$ belongs to $L^1$ which play a crucial role throughout this paper. In fact, this property ensures the exchange of  differentiation and integration.  Nevertheless, we have the following lemma.
	
	\begin{lemma}\label{lem:|y|^n-2m+kQu^p in L1}
		For any integer $0\leq k<2m$ and $x\in\mr^n$ fixed, 
		\begin{enumerate}
			\item  there holds 	$$\int_{\mr^n}\frac{Q(y)u^p(y)}{|x-y|^{n-2m+k}}\ud y<+\infty, $$
			\item for $R\gg2|x|+1$, there holds
			$$\int_{\mr^n\backslash B_R}\frac{Q(y)u^p(y)}{|x-y|^{n-2m+k}}\ud y\leq CR^{-k-\frac{2m+s}{p-1}}.$$
		\end{enumerate}
	\end{lemma}
	\begin{proof}
		Making use of  Lemma \ref{lem: B_R Qup}, there exists $R_1>0$ such that
		for any $R\geq R_1$,  one has 
		\begin{equation}\label{equ: 2^k+1-2^k}
			\int_{B_{2R}\backslash B_R}Q(y)u^p(y)\ud y\leq CR^{n-2m-\frac{2m+s}{p-1}}.
		\end{equation}
	Making use of such estimate and choosing $R>\max\{R_1,2|x|\}$, one obtains the following estimate
			\begin{equation}\label{equ: 2.5}
		\int_{B_{2R}\backslash B_R}\frac{Q(y)u^p(y)}{|x-y|^{n-2m+k}}\ud y\leq CR^{-k-\frac{2m+s}{p-1}}.
	\end{equation}
As a consequence, one has
		\begin{align*}
			&\int_{\mr^n\backslash B_R}\frac{Q(y)u^p(y)}{|x-y|^{n-2m+k}}\ud y\\
			\leq &\sum^\infty_{i=0}\int_{B_{2^{i+1}R}\backslash B_{2^iR}}\frac{Q(y)u^p(y)}{|x-y|^{n-2m+k}}\ud y\\
			\leq &C\sum^\infty_{i=0}(2^iR)^{-k-\frac{2m+s}{p-1}}\\
			\leq &CR^{-k-\frac{2m+s}{p-1}}.
		\end{align*}
		
		Since $n-2m+k<n$ as well as  both $Q$ and $u$ belonging to $L^\infty_{loc}$, it is not hard to verify that 
		$$\int_{B_R}\frac{Q(y)u^p(y)}{|x-y|^{n-2m+k}}\ud y< +\infty.$$
		Combining these two estimates, one has 
		$$\int_{\mr^n}\frac{Q(y)u^p(y)}{|x-y|^{n-2m+k}}\ud y<+\infty.$$
		
		Thus the  proof of this lemma is complete.
	\end{proof}
	
	Now we are in position to state and prove one important property for the solution $u$, namely we have the following lemma.

	\begin{lemma}\label{lem: u L^1 growth}
		For $R\gg1$, there holds
		$$\int_{B_R}u(x)\ud x=o(R^n)$$
		
	\end{lemma}
	\begin{proof}
		By slow decay barrier condition, there exist constants $R_2 > 0$  and  $C > 0$ such that,  for any  $|x|\geq R_2$, one has 
		\begin{equation}\label{Q lower bound}
			Q(x)\geq C |x|^{s}.
		\end{equation}
		Making use of Lemma \ref{lem: B_R Qup}, for $R\gg1$, together with this assumption \ref{Q lower bound}, one has
		\begin{align*}
			\int_{B_R}u\ud x=&\int_{B_{R_2}}u(x)\ud x+\int_{B_{R}\backslash B_{R_2}}u\ud x\\
			\leq &C+CR^{-\frac{s}{p}}\int_{B_{R}\backslash B_{R_2}}Q^{\frac{1}{p}}u\ud x\\
			\leq &C+CR^{-\frac{s}{p}}\int_{B_{R}}Q^{\frac{1}{p}}u\ud x\\
			\leq &C+CR^{-\frac{s}{p}}(\int_{B_{R}}Qu^p\ud x)^{\frac{1}{p}}R^{n(1-\frac{1}{p})}\\
			\leq &C+CR^{n-\frac{2m+s}{p-1}}\\
			=&o(R^n).
		\end{align*}
		This completes the proof. 
			\end{proof}

	It is well-known  that the fundamental solution for the polyharmonic operator $(-\Delta)^m$  on $\mr^n$ satisfies the following equation
	\begin{equation}\label{Green}
		(-\Delta)^{m}|x|^{2m-n}=C(n,m)\delta_0(x)
	\end{equation}
	where $\delta_0(x)$ is  the delta function and $C(n,m)$ is some positive constant depending on $m,n$. For more details, the interested readers are referred  to  \cite{ACL}. With help of Lemma \ref{lem:|y|^n-2m+kQu^p in L1}, in order to get integral representation for the solution $u$,  let us consider the following function $v$: $$v(x):=\frac{1}{C(n,m)}\int_{\mr^n}\frac{Q(y)u(y)^p}{|x-y|^{n-2m}}\ud y.$$
	
	The following lemma studies the some properties of the function $v$, which can be derived using Theorem 6.21 in \cite{LL}. For simplicity, the proof is omitted.
	
	\begin{lemma}\label{lem:potential}
		For $1\leq i\leq 2m-1$, there holds
		$$\nabla^iv(x)=\frac{1}{C(n,m)}\int_{\mr^n}\nabla_x^i|x-y|^{2m-n}Q(y)u(y)^p\ud y$$
		as well as 
		$$(-\Delta)^mv(x)=Q(x)u(x)^p.$$
	\end{lemma}
	The function $v$ and the solution $u$ have many common properties. 
In particular, the following lemma states one which will be used later.
	
	\begin{lemma}\label{lem: nabla v growth}
		For $R\gg1$ and $0\leq k\leq 2m-1$, there holds
		$$\int_{B_R}|\nabla^kv(x)|\ud x=O(R^{n-k-\frac{2m+s}{p-1}}).$$
	\end{lemma}
	
	\begin{proof}
	With help of Lemma \ref{lem:potential} and Lemma \ref{lem:|y|^n-2m+kQu^p in L1},  there holds 
		\begin{align*}
			&\int_{B_R}|\nabla^kv(x)|\ud x\\
			\leq &C\int_{B_R}\int_{\mr^n}\frac{Q(y)u(y)^p}{|x-y|^{n-2m+k}}\ud y\ud x\\
			\leq &C\int_{B_R}\int_{\mr^n\backslash B_{3R}}\frac{Q(y)u(y)^p}{|x-y|^{n-2m+k}}\ud y\ud x\\
			&+C\int_{B_R}\int_{ B_{3R}}\frac{Q(y)u(y)^p}{|x-y|^{n-2m+k}}\ud y\ud x\\
			\leq &CR^{n-k-\frac{2m+s}{p-1}}+C\int_{ B_{3R}}Q(y)u(y)^p\left(\int_{B_R}\frac{1}{|x-y|^{n-2m+k}}\ud x\right) \ud y\\
			\leq &CR^{n-k-\frac{2m+s}{p-1}}+C\int_{ B_{3R}}Q(y)u(y)^p\int_{B_{4R}}\frac{1}{|z|^{n-2m+k}}\ud z \ud y\\
			\leq &CR^{n-k-\frac{2m+s}{p-1}}.
		\end{align*}
		\end{proof}
	
	The following result has been established as Proposition 1.3 in \cite{NY}. Here, we give another proof based on the work of \cite{Mar}.
	
	\begin{theorem}\label{thm: u=v}
		Considering  the equation \eqref{equ:Lane-Emden equ} satisfying the condition stated as before, there holds
		$$u(x)=\frac{1}{C(n,m)}\int_{\mr^n}\frac{Q(y)u(y)^p}{|x-y|^{n-2m}}\ud y.$$
	\end{theorem}
	\begin{proof}
		Set $P:=u-v$. Lemma \ref{lem:potential} and the equation  \eqref{equ:Lane-Emden equ} imply $P$ is a polyharmonic function, that is,
		$$(-\Delta)^mP=0.$$
		With help of Lemma \ref{lem: u L^1 growth}, Lemma \ref{lem: nabla v growth} and the facts $u,v$ are positive, one has 
		\begin{equation}\label{P growth}
			\int_{B_R}|P|\ud x\leq \int_{B_{R}}u\ud x+\int_{B_R}v\ud x=o(R^n).
		\end{equation}
		
		We now claim that the estimate \eqref{P growth} guarantees that $P\equiv 0$. The argument follows a standard approach by applying Proposition 4 in [16] and proceeding similarly to the proof of Theorem 5 in [16]. For the readers' convenience, we provide a brief outline of the proof below.
		
		Let us recall a formula for a polyharmonic function (See \cite{Piz} or Lemma 3 in \cite{Mar}): for any $x\in \mr^n$ and $R>0$, one has
		\begin{equation}\label{PZ formular}
			\frac{1}{|B_R(x)|}\int_{B_R(x)} P(y)\ud y=\sum^{m-1}_{i=0}c_i R^{2i}\Delta^{i}P(x) 
		\end{equation}
		where $c_i$ are some positive dimensional constants.  Suppose there exists a $0 \leq k \leq m-1$ which is largest such that $\Delta^kP \not\equiv 0$. Since  $\Delta^kP \not\equiv 0$, there must be at least one point $x_0 \in \mr^n$ such that  $\Delta^kP(x_0) \not= 0$. Apply above formula with $x = x_0$ to get 
		\begin{equation}\label{PZ formular1}
			\frac{1}{|B_R(x_0)|}\int_{B_R(x_0)} P(y)\ud y=\sum^{k}_{i=0}c_i R^{2i}\Delta^{i}P(x_0). 
		\end{equation} 
		Now we divide the both sides by $R^{2k}$ and then take the limit as $R\to\infty$ in the formula \eqref{PZ formular1}, we see that the coefficient of the leading term must be zero i.e.
		$$\Delta^{k}P(x_0)=0.$$ 
			This will contradict with the assumption that $\Delta^k P (x_0) \not = 0$. Therefore our claim holds true.
	\end{proof}

	\section{Preparations}\label{section3}
	For simplicity,  we use the notation $\int_{\mr^{kn}}$ to denote the $k-$th multiple integral $\int\cdots\int.$  The various measures should mean 
	$$\ud \mu(y):=\frac{1}{C(n,m)}Q(y)u(y)^p\ud y \quad \mathrm{in}\;\;\mr^n$$
	and
	$$\ud\mu(y_1,y_2,\cdots,y_l ):=\ud \mu(y_1)\ud \mu(y_2)\cdots \ud \mu(y_l) \quad \mathrm{in}\;\;\mr^{ln}.$$

	Throughout this section, without special notification, we should consider the equation \eqref{equ:Lane-Emden equ} under slow decay barrier condition as mentioned before.
	For later references,  we define three non-negative functions as follows:
	\begin{equation}\label{A1}
		A_1(x) :=\int_{\mr^{2n}}\frac{|y-z|^2(|x-y|^2+|x-z|^2)}{|x-y|^{n-2m+4}|x-z|^{n-2m+4}}\ud\mu(y)\ud\mu(z).
	\end{equation}
	\begin{equation}\label{A_2}
		A_2(x) :=\int_{\mr^{3n}}\frac{|y-z|^2\left(|x-z|^2|y-s|^2+|x-y|^2|z-s|^2\right)}{|x-y|^{n-2m+4}|x-z|^{n-2m+4}|x-s|^{n-2m+2}}\ud\mu(y, z, s).
	\end{equation}
	as well as
	\begin{equation}\label{A_3}
		A_3(x):=\int_{\mr^{2n}}\frac{|y-z|^4}{|x-y|^{n-2m+4}|x-z|^{n-2m+4}}\ud\mu(y, z)
	\end{equation}
	
	In fact, Lemma \ref{lem:|y|^n-2m+kQu^p in L1} and Fubini's theorem ensure that $A_1, A_2, A_3$ are well-defined for each $x\in \mr^n.$
	
	Now, we set up a short notation $A(u)$ to indicate the right hand side in the following equation, that is,  
	\begin{equation}\label{A(u)}
		A(u):=\frac{1}{(m-1)(n-2m)}\left(-\frac{2m-2}{n-2m}|\nabla u|^2-u\Delta u\right).
	\end{equation}

In fact, $A(u)$ arises from the representation of the scalar curvature $R_g$ using \eqref{R_g=A(u)u^p}, which implies that the positivity of $R_g$ is equivalent to 
$A(u)>0$. Utilizing Fubini's theorem, 
$A(u)$  can be expressed through a positive  integral representation. Using a similar approach, we denote $B(u)$ for 
$Q^{(4)}_g$  multiplied by $u^{\frac{n+4}{n-2m}}$ up to a  positive constant.  Subsequently, 
$B(u)$ is represented in terms of 
$A(u)$ and $A_i$ , which were defined earlier. By leveraging the key observations 
$A(u)>0$ , $u^2A_3+A(u)^2-uA_2>0$ and $u^2A_3-A(u)^2>0$,  we establish that 
$Q^{(4)}_g$ is positive.

	Now we do some calculations in terms of those functions we just defined. Firstly,  we have a lemma.
	\begin{lemma}\label{lem:Delta u^t}
		For each real number $t\in \mr$, there holds
		$$\Delta u^t=t\left(1+\frac{(t-1)(2m-n)}{2m-2}\right)u^{t-1}\Delta u-\frac{t(t-1)}{2}(n-2m)^2u^{t-2}A(u).$$
	\end{lemma}
	\begin{proof}
		This is a direct computation. In fact, it is not hard to check the following computation:
		\begin{align*}
			\Delta u^t=&tu^{t-1}\Delta u+t(t-1)u^{t-2}|\nabla u|^2\\
			=&tu^{t-1}\Delta u+t(t-1)u^{t-2}\frac{2m-n}{2m-2}\left((m-1)(n-2m)A(u)+u\Delta u\right)\\
			=&t\left(1+\frac{(t-1)(2m-n)}{2m-2}\right)u^{t-1}\Delta u-\frac{t(t-1)}{2}(n-2m)^2u^{t-2}A(u).
		\end{align*}
	\end{proof}
	
	A similar computation gives the following formula.
	
	\begin{lemma}\label{lem:nabla unabla Delta u}
		For $3\leq m<\frac{n}{2}$, 	there holds
		\begin{align*}
			\nabla u\cdot\nabla \Delta u = & \frac{2m-n}{4(m-2)}u\Delta^2u\\
			& +\frac{2m-n-2}{4(m-1)}(\Delta u)^2+(n-2m)^2(m-1)(n+2-2m)\frac{A_1}{2}.
		\end{align*}
	\end{lemma}
	
	\begin{proof}
		Making use of Lemma \ref{lem:potential} and Theorem \ref{thm: u=v}, the following calculation is straightforward.  Since
		\begin{equation}\label{nabla u}
			\nabla u(x)=(2m-n)\int_{\mr^{n}}\frac{x-y}{|x-y|^{n-2m+2}}\ud\mu(y), 
		\end{equation}
		naturally one has:  
		\begin{align*}
			|\nabla u(x)|^2=&\sum_{i=1}^n\left((2m-n)\int_{\mr^n}\frac{x_i-y_i}{|x-y|^{n-2m+2}}\ud\mu(y)\right)^2\\
			=&(2m-n)^2\sum_{i=1}^n\int_{\mr^{n}}\frac{x_i-y_i}{|x-y|^{n-2m+2}}\ud\mu(y)\int_{\mr^{n}}\frac{x_i-z_i}{|x-z|^{n-2m+2}}\ud\mu(z)\\
			=&(2m-n)^2\int_{\mr^{2n}}\frac{(x-y)\cdot(x-z)}{|x-y|^{n-2m+2}|x-z|^{n-2m+2}}\ud\mu(y,z).
		\end{align*}
		Due to the fact that the order of differentiation and integration can be exchanged, one easily obtains  
		\begin{equation}\label{Delta u}
			\Delta u(x)=(2m-n)(2m-2)\int_{\mr^{n}}\frac{1}{|x-y|^{n-2m+2}}\ud \mu(y)
		\end{equation}
		as well as 
		$$\Delta^2u(x)=(2m-n)(2m-2)(2m-2-n)(2m-4)\int_{\mr^{n}}\frac{1}{|x-y|^{n-2m+4}}\ud \mu(y).$$
		By writing the iterated integration as the double integral,  $(\Delta u)^2$ can be rephrased as follows:
		$$(\Delta u(x))^2=(2m-n)^2(2m-2)^2\int_{\mr^{2n}}\frac{1}{|x-y|^{n-2m+2}|x-z|^{n-2m+2}}\ud\mu(y,z).$$
		
		Now first denote by $l$ the number $2m-n- 2$. Then combine Lemma \ref{lem:potential},  \eqref{nabla u} and \eqref{Delta u} together to see that $\nabla u\cdot\nabla \Delta u$ has the following integral representation:
		\begin{align*}
			&\nabla u\cdot\nabla \Delta u\\
			=&(2m-n)^2(2m-2)l\int_{\mr^{2n}}\frac{(x-y)\cdot(x-z)}{|x-y|^{n-2m+2}|x-z|^{n-2m+4}}\ud\mu(y,z)\\
			=&(2m-n)^2(m-1)l\int_{\mr^{2n}}\frac{|x-y|^2+|x-z|^2-|y-z|^2}{|x-y|^{n-2m+2}|x-z|^{n-2m+4}}\ud\mu(y,z)\\
			=&(2m-n)^2(m-1)l\int_{\mr^{2n}}\frac{1}{|x-y|^{n-2m}|x-z|^{n-2m+4}}\ud\mu(y,z)\\
			&+(2m-n)^2(m-1)l\int_{\mr^{2n}}\frac{1}{|x-y|^{n-2m+2}|x-z|^{n-2m+2}}\ud\mu(y,z)\\
			&-(2m-n)^2(m-1)l\int_{\mr^{2n}}\frac{|y-z|^2}{|x-y|^{n-2m+2}|x-z|^{n-2m+4}}\ud\mu(y,z)\\
			&= I + II + III.
		\end{align*}
		
		Now, we should deal with  terms on the right-hand side one by one. First term first, one has
		\begin{align*}
			 I	=&(2m-n)^2(m-1)l\int_{\mr^{n}}\frac{1}{|x-y|^{n-2m}}\ud\mu(y)\int_{\mr^{n}}\frac{1}{|x-z|^{n-2m+4}}\ud\mu(z)\\
			=&\frac{2m-n}{4(m-2)}u(x)\Delta^2u(x).
		\end{align*}
		Second term $II$, by using the same trick, can be rewritten as 
		\begin{align*}
			II =&(2m-n)^2(m-1)l\int_{\mr^{n}}\frac{1}{|x-y|^{n-2m+2}}\ud\mu(y)\int_{\mr^{n}}\frac{1}{|x-z|^{n-2m+2}}\ud\mu(z)\\
			=&\frac{2m-n-2}{4(m-1)}(\Delta u)^2.
		\end{align*}
		Finally, by doing variable changes for $y,z$, one has
		\begin{align*} 
			III &= -(2m-n)^2(m-1)l\int_{\mr^{2n}}\frac{|y-z|^2}{|x-z|^{n-2m+2}|x-y|^{n-2m+4}}\ud\mu(y,z).
		\end{align*}
	By symmetry of the integration, clearly there holds
		$$A_1(x) =2\int_{\mr^{2n}}\frac{|y-z|^2}{|x-y|^{n-2m+2}|x-z|^{n-2m+4}}\ud\mu(y,z).$$
			Combining these identities,  we obtain the desired identity:
		\begin{align*}
			\nabla u\cdot\nabla \Delta u = & \frac{2m-n}{4(m-2)}u\Delta^2u+\frac{2m-n-2}{4(m-1)}(\Delta u)^2\\
			& +(n-2m)^2(m-1)(n+2-2m)\frac{A_1}{2}.
		\end{align*}
	\end{proof}
	
	We also need the integral representation for function $A(u)$ which can be easily seen. For convenience, we record it as a lemma.
	
	\begin{lemma}\label{lem:Au)}
		There holds
		$$A(u)=\int_{\mr^{2n}}\frac{|y-z|^2}{|x-y|^{n-2m+2}|x-z|^{n-2m+2}}\ud\mu(y,z)>0.$$
	\end{lemma}
	\begin{proof}
		As we have shown in  \eqref{nabla u}, we have 
		$$|\nabla u|^2=(2m-n)^2\int_{\mr^{2n}}\frac{(x-y)\cdot(x-z)}{|x-y|^{n-2m+2}|x-z|^{n-2m+2}}\ud\mu(y)\ud\mu(z).$$
	By using Fubini's theorem and  the symmetry of $y$ and $z$, one has 
		\begin{align*}
			u\Delta u	=&(2m-n)(2m-2)\int_{\mr^{2n}}\frac{1}{|x-y|^{n-2m+2}|x-z|^{n-2m}}\ud \mu(y, z)\\
			=&(2m-n)(m-1)\int_{\mr^{2n}}\frac{|x-y|^2+|x-z|^2}{|x-y|^{n-2m+2}|x-z|^{n-2m+2}}\ud \mu(y, z).
		\end{align*}
		
		Since both terms $|\nabla u|^2$ and $u\Delta u$ have the integral representations, it is easy to see we have
		\begin{align*}
			&(m-1)(n-2m)A(u)\\
			=&-\frac{2m-2}{n-2m}|\nabla u|^2-u\Delta u\\
			=&(m-1)(n-2m)\int_{\mr^{2n}}\frac{|x-y|^2+|x-z|^2-2(x-y)\cdot(x-z)}{|x-y|^{n-2m+2}|x-z|^{n-2m+2}}\ud \mu(y, z)\\
			=&(m-1)(n-2m)\int_{\mr^{2n}}\frac{|y-z|^2}{|x-y|^{n-2m+2}|x-z|^{n-2m+2}}\ud \mu(y, z)
		\end{align*}
		which is the desired result.
	\end{proof}
	
	In the following calculation, we need the formulas for the Laplace of $A(u)$ as well as $\nabla u\cdot \nabla A(u)$ in terms of $A_i$ as well as $A(u)$. The following lemma serves this purpose.
	
	\begin{lemma}\label{lem:Delta A(u)}
		For $3\leq m<\frac{n}{2}$, there holds
		\begin{enumerate}
			\item {$\Delta A(u)=(2m-n-2)(4m-n-6)A_1-(2m-n-2)^2A_3,$}
			\item {$\nabla u\cdot\nabla A(u)=-\frac{n-2m+2}{2(m-1)}A(u)\Delta u
				+\frac{(n-2m)(n-2m+2)}{2}\left(uA_1
				-A_2\right).$}
		\end{enumerate}
	\end{lemma}
	
	\begin{proof}
		By the same reason as in Lemma \ref{lem:potential}, we can freely exchange the Laplacian operator and integration. With the elementary identity $$2(x-z)\cdot(x-y)=|x-z|^2+|x-y|^2-|y-z|^2,$$  and a direct computation, we can arrive at:
		\begin{align*}
			&\Delta A(u)\\
			=&\int_{\mr^{2n}}\Delta_x\left(|x-y|^{2m-n-2}\cdot|x-z|^{2m-n-2}\right)|y-z|^2\ud\mu(y,z)\\
			=&\int_{\mr^{2n}}(2m-n-2)(2m-4)|x-y|^{2m-n-4}\cdot|x-z|^{2m-n-2}|y-z|^2\ud\mu(y,z)\\
			&+\int_{\mr^{2n}}(2m-n-2)(2m-4)|x-y|^{2m-n-2}\cdot|x-z|^{2m-n-4}|y-z|^2\ud\mu(y,z)\\
			&+2(2m-n-2)^2\int_{\mr^{2n}}\frac{(x-y)\cdot(x-z)}{|x-y|^{n-2m+4}|x-z|^{n-2m+4}}|y-z|^2\ud\mu(y,z)\\
			=&(2m-n-2)(2m-4)A_1\\
			&+(2m-n-2)^2\int_{\mr^{2n}}\frac{|x-y|^2+|x-z|^2-|y-z|^2}{|x-y|^{n-2m+4}|x-z|^{n-2m+4}}|y-z|^2\ud\mu(y,z)\\
			=&(2m-n-2)(2m-4)A_1+(2m-n-2)^2A_1-(2m-n-2)^2A_3\\
			=&(2m-n-2)(4m-n-6)A_1-(2m-n-2)^2A_3.
		\end{align*}
		
		To continue, we denote the number $(n-2m)(n+2-2m)$ by $\alpha(n, m)$. Another straightforward calculation provides the information we need:
		\begin{align*}
			&\nabla u\cdot\nabla A(u)\\
			=&\int_{\mr^{3n}}\nabla_x|x-s|^{2m-n}\cdot\nabla_x(|x-y|^{2m-n-2}\cdot|x-z|^{2m-n-2})|y-z|^2\ud\mu(y,z,s)\\
			=&\alpha(n, m)\int_{\mr^{3n}}\frac{(x-y)\cdot(x-s)|x-z|^2|y-z|^2}{|x-y|^{n-2m+4}|x-z|^{n-2m+4}|x-s|^{n-2m+2}}\ud\mu(y,z,s)\\
			&+ \alpha(n, m)\int_{\mr^{3n}}\frac{(x-z)\cdot(x-s)|x-y|^2|y-z|^2}{|x-y|^{n-2m+4}|x-z|^{n-2m+4}|x-s|^{n-2m+2}}\ud\mu(y,z,s)\\
			=&\frac{\alpha(n, m)}{2}\int_{\mr^{3n}}\frac{(|x-y|^2+|x-s|^2-|y-s|^2)|x-z|^2|y-z|^2}{|x-y|^{n-2m+4}|x-z|^{n-2m+4}|x-s|^{n-2m+2}}\ud\mu(y,z,s)\\
			&+\frac{\alpha(n, m)}{2}\int_{\mr^{3n}}\frac{(|x-z|^2+|x-s|^2-|z-s|^2)|x-y|^2|y-z|^2}{|x-y|^{n-2m+4}|x-z|^{n-2m+4}|x-s|^{n-2m+2}}\ud\mu(y,z,s)\\
			=& \alpha(n, m)\int_{\mr^{3n}}\frac{|x-z|^2|x-y|^2|y-z|^2}{|x-y|^{n-2m+4}|x-z|^{n-2m+4}|x-s|^{n-2m+2}}\ud\mu(y,z,s)\\
			&+\frac{\alpha(n, m)}{2}\int_{\mr^{3n}}\frac{|y-z|^2\left(|x-z|^2|x-s|^2+|x-y|^2|x-s|^2\right)}{|x-y|^{n-2m+4}|x-z|^{n-2m+4}|x-s|^{n-2m+2}}\ud\mu(y,z,s)\\
			&-\frac{\alpha(n, m)}{2}\int_{\mr^{3n}}\frac{|y-z|^2\left(|x-z|^2|y-s|^2+|x-y|^2|z-s|^2\right)}{|x-y|^{n-2m+4}|x-z|^{n-2m+4}|x-s|^{n-2m+2}}\ud\mu(y,z,s).
		\end{align*}
		
		The first term of the right side can be simplified to obtain:
		\begin{align*}
			& \int_{\mr^{3n}}\frac{|x-z|^2|x-y|^2|y-z|^2}{|x-y|^{n-2m+4}|x-z|^{n-2m+4}|x-s|^{n-2m+2}}\ud\mu(y,z,s)\\
			=&\int_{\mr^{3n}}\frac{|y-z|^2}{|x-y|^{n-2m+2}|x-z|^{n-2m+2}|x-s|^{n-2m+2}}\ud\mu(y,z,s)\\
			=&\int_{\mr^{2n}}\frac{|y-z|^2}{|x-y|^{n-2m+2}|x-z|^{n-2m+2}}\ud\mu(y,z)\int_{\mr^{n}}\frac{1}{|x-s|^{n-2m+2}}\ud\mu(s)\\
			=&-\frac{1}{2(m-1)(n-2m)}A(u)\Delta u.
		\end{align*}
		Similarly, the second term will take the short form: 
		\begin{align*}
			&\int_{\mr^{3n}}\frac{|y-z|^2\left(|x-z|^2|x-s|^2+|x-y|^2|x-s|^2\right)}{|x-y|^{n-2m+4}|x-z|^{n-2m+4}|x-s|^{n-2m+2}}\ud\mu(y,z,s)\\
			=&\int_{\mr^{3n}}\frac{|y-z|^2(|x-z|^2+|x-y|^2)}{|x-y|^{n-2m+4}|x-z|^{n-2m+4}|x-s|^{n-2m}}\ud\mu(y,z,s)\\
			=&\int_{\mr^{2n}}\frac{|y-z|^2(|x-z|^2+|x-y|^2)}{|x-y|^{n-2m+4}|x-z|^{n-2m+4}}\ud\mu(y,z)\int_{\mr^{n}}\frac{1}{|x-s|^{n-2m}}\ud\mu(s)\\
			=&uA_1.
		\end{align*}
		
		Based on the definition of the function $A_2(x)$ (see the formula \eqref{A_2}),  the third term is exactly $-\frac{\alpha(n, m)}{2}A_2$.
		
		Combining these identities, one finally obtains:
		$$\nabla u\cdot\nabla A(u)=-\frac{n-2m+2}{2(m-1)}A(u)\Delta u
		+\frac{(n-2m)(n-2m+2)}{2}\left(uA_1
		-A_2\right).$$
		\end{proof}
	
	Before we end this section, we do two technical integrations which are needed in the proof of our main result. The next lemma is to deal with the first integral.
	
	\begin{lemma}\label{lem:P_1/P_2}
		For $3\leq m<\frac{n}{2}$, there holds
		$$\int_{\mr^{4n}}\frac{P_1(x,y,z,s,w)}{P_2(x,y,z,s,w)}\ud\mu(y,z,s,w)=4u^2A_3+4A(u)^2-4uA_2$$
		where
		\begin{align*}
			P_1(x,y,z,s,w)=&(|x-y|^2|x-z|^2|s-w|^2-|x-y|^2|x-w|^2|z-s|^2\\
			&+|x-s|^2|x-w|^2|y-z|^2-|x-z|^2|x-s|^2|y-w|^2)^2
		\end{align*}
		and 
		$$P_2(x,y,z,s,w)=|x-y|^{n-2m+4}|x-z|^{n-2m+4}|x-s|^{n-2m+4}|x-w|^{n-2m+4}.$$
	\end{lemma}
	\begin{proof}
		By elementary multiplication, we have the following identity:
		\begin{align*}
			P_1(x,y,z,s,w)=&(|x-y|^2|x-z|^2|s-w|^2-|x-y|^2|x-w|^2|z-s|^2\\
			&+|x-s|^2|x-w|^2|y-z|^2-|x-z|^2|x-s|^2|y-w|^2)^2\\
			=&|x-y|^4|x-z|^4|s-w|^4\\
			&+|x-y|^4|x-w|^4|z-s|^4\\
			&+|x-s|^4|x-w|^4|y-z|^4\\
			&+|x-z|^4|x-s|^4|y-w|^4\\
			&+2|x-y|^2|x-z|^2|x-s|^2|x-w|^2|s-w|^2|y-z|^2\\
			&+2|x-y|^2|x-z|^2|x-s|^2|x-w|^2|z-s|^2|y-w|^2\\
			&-2|x-y|^4|x-z|^2|x-w|^2|s-w|^2|z-s|^2\\
			&-2|x-z|^4|x-y|^2|x-s|^2|s-w|^2|y-w|^2\\
			&-2|x-w|^4|x-y|^2|x-s|^2|z-s|^2|y-z|^2\\
			&-2|x-s|^4|x-z|^2|x-w|^2|y-z|^2|y-w|^2\\
			& = \sum_{k = 1}^{10} D_k.
		\end{align*}
		With help of this identity, the integral will decompose into the ten terms, namely,
		\begin{align*}
			& \int_{\mr^{4n}}\frac{P_1(x,y,z,s,w)}{P_2(x,y,z,s,w)}\ud\mu(y,z,s,w)\\
			= &\sum_{k=1}^{10} \int_{\mr^{4n}}\frac{D_k}{P_2(x,y,z,s,w)}\ud\mu(y,z,s,w)\\
			:= & \sum_{k=1}^{10} I_k.
		\end{align*}

		Now we treat term by term on the right. By definition of $P_2$, the first integral can be evaluated into:
		\begin{align*}
			I_1 = & \int_{\mr^{4n}}\frac{|x-y|^4|x-z|^4|s-w|^4}{P_2(x,y,z,s,w)}\ud\mu(y,z,s,w)\\
			=&\int_{\mr^{4n}}\frac{|s-w|^4|x-y|^{2m-n}|x-z|^{2m-n}}{|x-s|^{n-2m+4}|x-w|^{n-2m+4}}\ud\mu(y,z,s,w)\\
			=&\int_{\mr^{n}}\frac{1}{|x-y|^{n-2m}}\ud\mu(y)\int_{\mr^{n}}\frac{1}{|x-z|^{n-2m}}\ud\mu(z)\\
			\cdot & \int_{\mr^{2n}}\frac{|s-w|^{4}}{|x-s|^{n-2m+4}|x-w|^{n-2m+4}}\ud\mu(s,w)\\
			=&u^2(x)A_3(x).
		\end{align*}
		
		By the symmetry of variables $y, z, s, w$, similar to $I_1$,  it is easy to see that $I_2 = I_3 = I_4 = u^2A_3$.

		The fifth integral $I_5$ can be treated as follows: 
		\begin{align*}
			&2\int_{\mr^{4n}}\frac{|x-y|^2|x-z|^2|x-s|^2|x-w|^2|s-w|^2|y-z|^2}{P_2(x,y,z,s,w)}\ud\mu(y,z,s,w)\\
			=&2\int_{\mr^{4n}}\frac{|s-w|^2|y-z|^2|x-y|^{2m-n-2}}{|x-z|^{n-2m+2}|x-s|^{n-2m+2}|x-w|^{n-2m+2}}\ud\mu(y,z,s,w)\\
			=&2\left(\int_{\mr^{2n}}\frac{|s-w|^2}{|x-s|^{n-2m+2}|x-w|^{n-2m+2}}\ud\mu(s,w)\right)\\
			\cdot & \left(\int_{\mr^{2n}}\frac{|y-z|^2}{|x-y|^{n-2m+2}|x-z|^{n-2m+2}}\ud\mu(y,z)\right)\\
			=&2A(u)^2(x) = I_5.
		\end{align*}
		
		Again by the symmetry of variables $y, z, s, w$, the sixth integral $I_6$ also equals to $2A(u)^2$.

		Before dealing with the last four integrals, we first recall the definition $A_2(x)$ as given in the formula \eqref{A_2} and rewrite it as 
	$$
			A_2(x)=\int_{\mr^{3n}}\frac{|y-z|^2\left(|x-z|^2|y-s|^2+|x-y|^2|z-s|^2\right)}{|x-y|^{n-2m+4}|x-z|^{n-2m+4}|x-s|^{n-2m+2}}\ud\mu(y, z, s).
	$$
	Using the symmetry of variables, $A_2(x)$ can be expressed as  
		\begin{equation}\label{A_2=integral}
			A_2(x)=2\int_{\mr^{3n}}\frac{|y-z|^2|y-s|^2|x-y|^{2m-n-4}}{|x-z|^{n-2m+2}|x-s|^{n-2m+2}}\ud\mu(y, z, s).
		\end{equation}
		
		Finally,  with help of Fubini's theorem and \eqref{A_2=integral}, the seventh integral $I_7$ can be reduced to:
		\begin{align*}
			&-2\int_{\mr^{4n}}\frac{|x-y|^4|x-z|^2|x-w|^2|s-w|^2|z-s|^2}{P_2(x,y,z,s,w)}\ud\mu(y,z,s,w)\\
			=&-2\int_{\mr^{4n}}\frac{|s-w|^2|z-s|^2|x-y|^{2m - n}}{|x-z|^{n-2m+2}|x-s|^{n-2m+4}|x-w|^{n-2m+2}}\ud\mu(y,z,s,w)\\
			=&-2 u(x) \int_{\mr^{3n}}\frac{|s-w|^2|z-s|^2}{|x-z|^{n-2m+2}|x-s|^{n-2m+4}|x-w|^{n-2m+2}}\ud\mu(z,s,w)\\
			=&-u(x)A_2(x) = I_7.
		\end{align*}
		By similar argument, we can get$I_8 = I_9 = I_{10} = -u(x)A_2(x).$
		
		Combining these identities, we finally conclude that
		$$\int_{\mr^{4n}}\frac{P_1(x,y,z,s,w)}{P_2(x,y,z,s,w)}\ud\mu(y,z,s,w) = [4u^2A_3+4A(u)^2-4uA_2](x).$$
	\end{proof}
	
	The second integral we need to handle is the following one.  The proof is straightforward  and we omit it.
	
	\begin{lemma}\label{lem: A(u)^2-uA_3}
		For $3\leq m<\frac{n}{2}$,
		there holds
		$$\frac{1}{2}\int_{\mr^{4n}}\frac{P_3(x,y,z,s,w)}{P_2(x,y,z,s,w)}\ud\mu(y,z,s,w) = [u^2A_3-A(u)^2](x),$$
		where
		$$P_3(x,y,z,s,w)=\left(|x-s|^2|x-w|^2|y-z|^2-|x-y|^2|x-w|^2|s-w|^2\right)^2.$$
	\end{lemma}

	\section{Proof of main result}\label{section4}
	
	We will divide the proof into two steps. 
	\vskip .2in
	
	{\bf Step 1.} In this place, our aim is to show that the scalar curvature of the metric $g$ is strictly positive.
	
	In fact,  this is rather easy. Notice that for a given conformal metric $g=u^{\frac{4}{n-2m}}|dx|^2$, 
	the scalar curvature can be calculated as
	\begin{align}\label{scalar curvature}
		R_g  &=-\frac{4(n-1)}{n-2}u^{-\frac{n+2}{n-2m}}\Delta u^{\frac{n-2}{n-2m}}\\
		 &=-\frac{4(n-1)}{n-2m}u^{\frac{4m-4-2n}{n-2m}}\left(\frac{2m-2}{n-2m}|\nabla u|^2+u\Delta u\right).\nonumber
	\end{align}
	With help of Lemma \ref{lem:Au)},	it is not hard to see that the formula \eqref{scalar curvature} yields that 
	\begin{equation}\label{R_g=A(u)u^p}
		R_g = 4(n-1)(m-1)u^{\frac{4m-4-2n}{n-2m}}A(u) > 0.
	\end{equation}

	{\bf Step 2.} We are now in the position to show that $Q_g^{(4)}$ is also strictly positive. Notice that when $m = 2$, we have nothing to do since the conclusion is just the assumption. In the next, we will always assume $m \geq 3$.
	
	Since we try to deal with $Q_g^{(4)}$, we write $t=\frac{n-4}{n-2m}$ throughout this step.
	
	By Lemma \ref{lem:Delta u^t}, we have
	\begin{equation}\label{Delat u^n-4/n-2m}
		\Delta u^t=\frac{t}{m-1}\left(u^{t-1}\Delta u-(m-2)(n-2m)(m-1)u^{t-2}A(u)\right).
	\end{equation}
	With help of the equations \eqref{Q^2k} and \eqref{Delat u^n-4/n-2m}, the fourth Q-curvature $Q_g^{(4)}$ can be calculated as:
	\begin{equation}\label{Q_4 }
		Q^{(4)}_g=u^{-\frac{n+4}{n-2m}}\Delta\left(\Delta u^t\right)=(m-2)(n-4)u^{-\frac{n+4}{n-2m}}B(u)
	\end{equation}
	where
	$$B(u):=\frac{\Delta(u^{t-1}\Delta u)}{(m-2)(n-2m)(m-1)}-\Delta(u^{t-2}A(u))$$
	To show that  $Q^{(4)}_g$ is positive, we just need to show that $B(u)$ is strictly positive.

	First of all, Lemma \ref{lem:Delta u^t} implies that, recall our $t$ is equal to $\frac{n-4}{n-2m}$,
	\begin{equation}\label{Delta u^t-1}
		\Delta u^{t-1}=-(m-2)(4m-4-n)u^{t-3}A(u)+\frac{(m-2)(n-2m+2)}{(n-2m)(m-1)}u^{t-2}\Delta u.
	\end{equation}
	
	Now this identity, together with Lemma \ref{lem:nabla unabla Delta u}, provides the formula:
	\begin{align*}
		\Delta(u^{t-1}\Delta u)=&(\Delta u^{t-1})\Delta u+u^{t-1}\Delta^2u+2(t-1)u^{t-2}\nabla u\cdot\nabla\Delta u\\
		=&-(m-2)(4m-4-n)u^{t-3}A(u)\Delta u \\
		& +\frac{(m-2)(n-2m+2)}{(n-2m)(m-1)}u^{t-2}(\Delta u)^2\\
		&+u^{t-1}\Delta^2u-u^{t-1}\Delta^2u+\frac{(m-2)(2m-n-2)}{(m-1)(n-2m)}u^{t-2}(\Delta u)^2\\
		&+2(m-2)(m-1)(n-2m)(n+2-2m)u^{t-2}A_1\\
		=&-(m-2)(4m-4-n)u^{t-3}A(u)\Delta u\\
		&+2(m-2)(m-1)(n-2m)(n+2-2m)u^{t-2}A_1.
	\end{align*}
	
	Once again Lemma \ref{lem:Delta u^t} provides the formula:
	\begin{equation}\label{Delta u^t-2}
		\Delta u^{t-2}=(4m-4-n)[-(3m-2-n)u^{t-4}A(u)+\frac{(n-2m+1)}{(n-2m)(m-1)}u^{t-3}\Delta u].
	\end{equation}
	Therefore Lemma \ref{lem:Delta A(u)} can be used to do the following calculation: 
	\begin{align*}
		& \Delta(u^{t-2}A(u))\\
		=&(\Delta u^{t-2})A(u)+u^{t-2}\Delta A(u)+2(t-2)u^{t-3}\nabla u\cdot\nabla A(u)\\
		=&\frac{(4m-4-n)(n-2m+1)}{(n-2m)(m-1)}u^{t-3}A(u)\Delta u \\
		&-(4m-4-n)(3m-2-n)u^{t-4}A(u)^2\\
		&+u^{t-2}\left((2m-n-2)(4m-n-6)A_1-(2m-n-2)^2A_3\right)\\
		&+2(t-2)u^{t-3}\frac{n-2m +2}{2(m-1)}\left(-A(u)\Delta u
		+(m-1)(n-2m)\left(uA_1
		-A_2\right)\right)\\
		=&\frac{n-4m+4}{(n-2m)(m-1)}u^{t-3}A(u)\Delta u-(4m-4-n)(3m-2-n)u^{t-4}A(u)^2\\
		&+2(n+2-2m)u^{t-2}A_1\\
		&-(4m-4-n)(n+2-2m)u^{t-3}A_2\\
		&-(2m-n-2)^2u^{t-2}A_3.
	\end{align*}
	Finally, put them together to conclude that
	\begin{align*}
		B(u)=&\frac{\Delta(u^{t-1}\Delta u)}{(m-2)(n-2m)(m-1)}-\Delta(u^{t-2}A(u))\\
		=&\frac{n+4-4m}{(n-2m)(m-1)}u^{t-3}A(u)\Delta u +2(n+2-2m)u^{t-2}A_1\\
		&-\frac{n-4m+4}{(n-2m)(m-1)}u^{t-3}A(u)\Delta u -2(n+2-2m)u^{t-2}A_1\\
		&+(4m-4-n)(3m-2-n)u^{t-4}A(u)^2\\
		&+(4m-4-n)(n+2-2m)u^{t-3}A_2\\
		&+(2m-n-2)^2u^{t-2}A_3\\
		=&(4m-4-n)(3m-2-n)u^{t-4}A(u)^2\\
		&+(4m-4-n)(n+2-2m)u^{t-3}A_2\\
		&+(n+2-2m)^2u^{t-2}A_3
	\end{align*}
	
	Now denote $n-2m+2$ by $l$  and  rewrite $B(u)$ as:
	\begin{equation}\label{B(u)}
		B(u) =(2m-2-l)(m-l)u^{t-4}A(u)^2+l(2m-2-l)u^{t-3}A_2+l^2u^{t-2}A_3.
	\end{equation}
	
	Based on the assumptions $n>2m$ and $m\geq 3$, it is rather easy to see that $l>2$ as well as  $m<2m-2$. 
	
	The positivity of $B(u)$  can be seen according to the range of $l$ in three different cases:
	
	{\bf Case 1:}  $2 < l \leq m$. 
	
	In this case, it is easy to check that the first two terms of \eqref{B(u)} are non-negative and the last term is strictly positive. Hence $B(u)$ is strictly positive. 
	
	{\bf Case 2:} $m< l\leq 2m-2$.
	
	In this case, to see $B(u)$ positive, we have to rewrite it as
	\begin{align*}
		B(u)=&l(2m-2-l)u^{t-3}A_2\\
		&+(l-m)(2m-2-l)u^{t-4}\left(u^2A_3-A(u)^2\right)\\
		&+\left(2l^2-(3m-2)l+2m^2-2m\right)u^{t-2}A_3.
	\end{align*}
	With help of Lemma \ref{lem: A(u)^2-uA_3}, we find that the first and second terms of the right side are non-negative. 
	It is not hard to check that $$2l^2-(3m-2)l+2m^2-2m>0$$ for $m<l\leq 2m-2$. Making use of  this fact, the third term of the right side is strictly positive.  Thus we show that $B(u)$ is strictly positive.

	{\bf Case 3:} $l>2m-2$.
	
	We rewrite $B(u)$ as follows 
	\begin{align*}
		B(u)=&l(l-2m+2)u^{t-4}\left(u^2A_3+A(u)^2-uA_2\right)\\
		&+m(l-2m+2)u^{t-4}\left(u^2A_3-A(u)^2\right)\\
		&+\left((m-2)l+2m^2-2m\right)u^{t-2}A_3.
	\end{align*}
	Since the coefficient of each term is positive, Lemmas \ref{lem:P_1/P_2} and \ref{lem: A(u)^2-uA_3} imply that the first two terms are strictly positive, clearly so is the third term. That is enough to see that $B(u)$ is strictly positive.

	The proof is complete.

\end{document}